\documentclass{amsart}
\usepackage{amsmath,amssymb,graphicx,amsthm,tabularx,bbm,array}
\usepackage{cite}
\usepackage{hyperref}
\usepackage{xcolor}

\newtheorem{theorem}{Theorem}
\newtheorem{lemma}[theorem]{Lemma}

\newtheorem{proposition}[theorem]{Proposition}

\newtheorem{definition}[theorem]{Definition}

\usepackage{cite}
\usepackage{hyperref}
\begin{document}

\title[spectral
radii of three kinds of tensors]{The extremal spectral
radii of $k$-uniform supertrees}

\author{Honghai Li}
\address{College of Mathematics and Information Science\\
Jiangxi Normal University\\
Nanchang, 330022,  China}
\email{lhh@jxnu.edu.cn}

\author{Jiayu Shao}
\address{Department of Mathematics\\
Tongji University\\
Shanghai, China} \email{jyshao@tongji.edu.cn}

\author{Liqun Qi}
\address{Department of Applied Mathematics\\
The Hong Kong Polytechnic University\\
Hung Hom, Kowloon, Hong Kong}
\email{liqun.qi@polyu.edu.hk}


\thanks{The first author's work was supported by National Natural
Science Foundation of China (No. 11201198),
Natural Science Foundation of
        Jiangxi Province (No. 20132BAB201013), the Sponsored Program for Cultivating Youths of
         Outstanding Ability in Jiangxi Normal
        University and his work was partially done when he
was visiting The Hong Kong Polytechnic University.
 The second author's work was supported by National Natural
Science Foundation of China (No. 11231004 and 11271288). The third author's work was supported by the Hong Kong Research Grant Council (Grant No. PolyU 502510, 502111, 501212
and 501913). }

\begin{abstract}
In this paper, we study some extremal problems of three kinds of spectral radii of $k$-uniform hypergraphs (the adjacency spectral radius, the signless Laplacian spectral radius and the incidence $Q$-spectral radius).
 We call a connected and acyclic
$k$-uniform hypergraph a supertree. We introduce the operation of ``moving edges" for hypergraphs, together with the two special cases of this operation: the edge-releasing operation and the total grafting operation. By studying the perturbation of these kinds of spectral radii of hypergraphs under these operations, we prove that for all these three kinds of spectral radii, the hyperstar $\mathcal{S}_{n,k}$
attains uniquely the maximum spectral radius among all $k$-uniform supertrees on $n$ vertices. We also determine the unique $k$-uniform supertree on $n$ vertices with the second largest spectral radius (for these three kinds of spectral radii). We also prove that for all these three kinds of spectral radii, the loose path $\mathcal{P}_{n,k}$ attains uniquely the minimum spectral radius among all $k$-th power hypertrees of $n$ vertices.  Some bounds on the incidence $Q$-spectral radius are given. The relation between the incidence $Q$-spectral radius and the spectral radius of the matrix product of the incidence matrix and its transpose is discussed.

\end{abstract}

\maketitle

\section{Introduction}

The spectral method has been proved successful in studying the structural properties of graphs and in solving those problems arising from combinatorics or so which do not seemingly refer to any eigenvalues. The introduction of tensor eigenvalues by Qi \cite{Qi05} and Lim \cite{Lim05} independently makes the generalization of the spectral technique from graphs to hypergraphs become possible.

In 2008, Lim \cite{Lim08} initially proposed to study spectral hypergraph theory via eigenvalues of tensors.  In 2012, Cooper and Dutle \cite{CoopDut12} defined the eigenvalues (and the spectrum) of a uniform hypergraph as the eigenvalues (and the spectrum) of the adjacency tensor of that hypergraph, and proved a number of interesting results on the spectra of hypergraphs, as well as some natural analogs of basic results in spectral graph theory. The (adjacency) spectrum of a uniform hypergraph were further studied in \cite{PearsonZhang, XieChang13adjaZ}.
In the same year, Hu and Qi \cite{HuQi12} proposed a definition for the Laplacian tensor of an even uniform
hypergraph, and analyzed its connection with edge and vertex connectivity.
Li, Qi, and Yu \cite{LiQiYu13} proposed another definition
for the Laplacian tensor of an even uniform hypergraph, established a variational
formula for its second smallest $Z$-eigenvalue, and used it to provide lower bounds
for the bipartition width of the hypergraph. Further, Qi \cite{Qi-Lap-signlessLap} proposed a simple and natural definition for the Laplacian
tensor $\mathcal{L}$ and the signless Laplacian tensor $\mathcal{Q}$ as $\mathcal{L}=\mathcal{D}-\mathcal{A}$ and $\mathcal{Q}=\mathcal{D}+\mathcal{A}$ respectively, where $\mathcal{A}$ is the adjacency tensor of the hypergraph (defined as in \cite{CoopDut12}), and $\mathcal{D}$ is the degree diagonal tensor of the hypergraph. The properties of these Laplacian  and signless Laplacian tensors were further studied in \cite{HuQi14,HuQishao13,HuQiXie,Qishaowang14,shaoshanwu-manu}.
 Following this, Hu and Qi \cite{HuQi13normLap} proposed the normalized Laplacian tensor  and made some explorations on it.

In \cite{XieChang13signLapZ, XieChang13signLapHeig}, Xie and Chang proposed a different
definition for the signless Laplacian tensor of an even uniform hypergraph, and studied its
largest and smallest $H$-eigenvalues and $Z$-eigenvalues.

In this paper, a connected and acyclic $k$-uniform hypergraph is called a supertree.

The concept of  power hypergraphs was introduced in \cite{HuQishao13}. Let $G=(V,E)$ be an ordinary graph. For every $k\geq3$, the $k$th power of $G$, $G^k:=(V^k,E^k)$ is defined as the $k$-uniform hypergraph with the edge set $E^k:=\{e\cup\{i_{e,1},\ldots, i_{e,k-2},\}|e\in E\}$ and the vertex set  $V^k:=V\cup(\cup_{e\in E}\{i_{e,1},\ldots, i_{e,k-2}\})$.
The $k$th power of an ordinary tree was called a hypertree there. By definition, the $k$th power of an ordinary tree is a supertree defined here.

In this paper, we study three kinds of spectral radii of $k$-uniform hypergraphs: the adjacency spectral radius, the signless Laplacian spectral radius and the incidence $Q$-spectral radius. We study some extremal problems of these three kinds of spectral radii for the class of $k$-uniform supertrees on $n$ vertices, and the class of $k$th power hypertrees on $n$ vertices.

The incidence matrix of a $k$-uniform hypergraph $G$ was introduced by Berge \cite{Berge-hypergraph}. The
tensor product, in the sense of \cite{Shao-tensorproduct,
Buchangjiang14}, of the incidence matrix $R$, the identity tensor $\mathbb{I}$ and the transpose of the incidence matrix is called the incidence $Q$-tensor, and is denoted by $\mathcal{Q}^{*}(G)$ (or simply $\mathcal{Q}^{*}$) in this paper. The spectral radius of the
incidence $Q$-tensor is called the incidence $Q$-spectral radius, of that
$k$-uniform hypergraph.   The incidence $Q$-tensor $\mathcal{Q}^{*}$ coincides with
the ``signless Laplacian tensor" proposed by Xie and Chang in \cite{XieChang13signLapZ, XieChang13signLapHeig}, for even uniform hypergraphs (which is different from the signless Laplacian tensor $\mathcal{Q}=\mathcal{D}+\mathcal{A}$ studied in this paper).

For the purpose of studying the extremal problems of that three kinds of spectral radii, we introduce the operation of ``moving edges" for hypergraphs, together with the two special cases of this operation: the edge-releasing operation and the total grafting operation. We study the perturbation of these three kinds of spectral radii of hypergraphs under these operations, and show that all these three kinds of spectral radii of supertrees strictly increase under the edge-releasing operation and the inverse of the total grafting operation.

Using these perturbation results, we prove that for all these three kinds of spectral radii, the hyperstar $\mathcal{S}_{n,k}$ attains uniquely the maximum spectral radius among all $k$-uniform supertrees on $n$ vertices. We also determine the unique $k$-uniform supertree on $n$ vertices with the second largest spectral radius (for these three kinds of spectral radii). Meanwhile, the corresponding minimization problems for these three kinds of spectral radii of supertrees are investigated, and we show that  for all these three kinds of spectral radii, the loose path $\mathcal{P}_{n,k}$ attains uniquely the minimum spectral radius among all $k$-th power hypertrees of $n$ vertices.

This paper is organized as follows. In Section~\ref{sec-prel}, notation and some definitions about
tensors and hypergraphs are given. In Section~\ref{sec-tree},
supertrees are defined and some properties of supertrees are discussed.
In Section~\ref{sec-main}, the
incidence $Q$-tensor $\mathcal{Q}^*$ of a uniform  hypergraph $G$ is defined as the tensor
product $\mathcal{Q}^*=R\mathbb{I}R^T$, where $R$ is the incidence matrix of $G$, and $\mathbb{I}$ is the identity tensor. We show in Section~\ref{sec-main} that the incidence $Q$-tensor is irreducible if and only if the associated hypergraph is connected. In Section~\ref{sec-lar}, we introduce the above-mentioned three operations on hypergraphs, and investigate the
perturbation of the three kinds of spectral radii of the supertrees under these operations. As applications, some extremal spectral problems
are solved. Two of the main results in Section~\ref{sec-lar} are the following theorems:

\begin{theorem}\label{thm-hyperstar0}
Let $\mathfrak{T}$ be a  $k$-uniform   supertree on $n$ vertices with $m$ edges.   Then
\begin{equation*}
\rho(\mathcal{A}(\mathfrak{T}))\leq m^{1/k},
\end{equation*}
and
\begin{equation*}
\rho(\mathcal{Q}(\mathfrak{T}))\leq 1+\alpha^*,
\end{equation*}
where $\alpha^*\in (m-1, m]$ is the largest real root of $x^k-(m-1)x^{k-1}-m=0$, and
\begin{equation*}
 \rho(\mathcal{Q}^*(\mathfrak{T}))\leq (m^{1/(k-1)}+k-1)^{k-1}
\end{equation*}
where either one of the equalities holds if and only if $\mathfrak{T}$ is the hyperstar $\mathcal{S}_{n,k}$.
\end{theorem}

\vskip 0.18cm

We also determine the unique $k$-uniform supertree on $n$ vertices with the second largest spectral radius (for these three kinds of spectral radii).

\vskip 0.18cm

\begin{theorem}\label{thm-kpowertree0}
Let  $T^k$ be the $k$th power of an ordinary tree $T$, defined as in \cite{HuQishao13}. Suppose that $T^k$ has $n$ vertices. Then we have
\begin{equation*}
\rho(\mathcal{A}(\mathcal{P}_{n,k}))\leq \rho(\mathcal{A}(T^k))\leq \rho(\mathcal{A}(\mathcal{S}_{n,k}))
\end{equation*}
and
\begin{equation*}
\rho(\mathcal{Q}(\mathcal{P}_{n,k}))\leq \rho(\mathcal{Q}(T^k))\leq \rho(\mathcal{Q}(\mathcal{S}_{n,k}))
\end{equation*}
and
\begin{equation*}
\rho(\mathcal{Q}^{*}(\mathcal{P}_{n,k}))\leq \rho(\mathcal{Q}^{*}(T^k))\leq \rho(\mathcal{Q}^{*}(\mathcal{S}_{n,k}))
\end{equation*}
where either one of the left equalities holds if and only if  $T^k\cong \mathcal{P}_{n,k}$, and either one of the right equalities holds if and only if  $T^k\cong \mathcal{S}_{n,k}$.
\end{theorem}

In the last section, some bounds on the incidence $Q$-spectral radius are presented, and the relation between the incidence $Q$-spectral radius and the spectral radius of the matrix product of the incidence matrix and its transpose is discussed.

\section{Preliminaries}\label{sec-prel}

A $k$th-order $n$-dimensional real tensor $\mathcal{T}$ consists of
$n^k$ entries in real numbers:
\[
\mathcal{T}=(\mathcal{T}_{i_1i_2\cdots
i_k}),\,\,\mathcal{T}_{i_1i_2\cdots i_k}\in \mathbb{R},\,\, 1\leq
i_1,i_2,\ldots,i_k\leq n.
\]
$\mathcal{T}$  is called \textit{symmetric} if the value of
$\mathcal{T}_{i_1i_2\cdots i_k}$ is invariant under any
permutation of its indices $i_1,i_2,\ldots,i_k$.  A real symmetric tensor $\mathcal{T}$ of order $k$ dimension $n$ uniquely defines a
$k$th degree homogeneous polynomial function $f$ with  real
coefficient by $f(x)=\mathcal{T}x^k$, which is a real scalar defined as
$$
\mathcal{T}x^k=\sum_{i_1,\ldots, i_k=1}^n\mathcal{T}_{i_1 \cdots i_k}x_{i_1}\cdots
x_{i_k}.
$$
$\mathcal{T}$ is  called positive semi-definite if $f(x)=\mathcal{T}x^k\geq0$ for all $x\in \mathbb{R}^n$. Clearly, for the nontrivial case, $k$ must be even.

Recall the definition of tensor product, $\mathcal{T}x^{k-1}$ is a vector in $\mathbb{R}^n$ with its $i$th
component as
\begin{equation}\label{e-produdefn}
(\mathcal{T}x^{k-1})_i=\sum_{i_2,\ldots,i_k=1}^n\mathcal{T}_{i i_2 \cdots i_k}x_{i_2}\cdots
x_{i_k}.
\end{equation}

\begin{definition}[\cite{Qi05}]
Let $\mathcal{T}$ be a $k$th-order $n$-dimensional tensor and $\mathbb{C}$ be the set of all complex numbers.   Then $\lambda$ is an eigenvalue of
$\mathcal{T}$ and $0\neq x\in \mathbb{C}^n$ is an
eigenvector corresponding to $\lambda$ if
$(\lambda,x)$ satisfies
\begin{equation*}
\mathcal{T}x^{k-1}=\lambda x^{[k-1]},
\end{equation*}
where $x^{[k-1]}\in \mathbb{C}^n$ with
$(x^{[k-1]})_i=(x_i)^{k-1}$.
\end{definition}

Several kinds of eigenvalues of tensors were defined in \cite{Qi05} and we focus on the one above in this paper.

A hypergraph
$G$ is a pair $(V,E)$, where $E\subseteq \mathcal{P}(V)$ and $\mathcal{P}(V)$ stands for the power set of $V$. The elements of $V=V(G)$, labeled as $[n]=\{1,\ldots,n\}$, are referred to as vertices and the elements of $E=E(G)$ are called edges. A
hypergraph $G$ is said to be $k$-uniform for an integer $k\geq2$ if, for all $e\in E(G)$,  $|e|=k$.  For a subset $S\subset V$, we denote by $E_S$ the set of edges $\{e\in E\ |\ S\cap e\neq\emptyset\}$. For  a vertex $i\in V$, we simplify $E_{\{i\}}$ as $E_i$. It is the set of edges containing the vertex $i$, i.e., $E_i=\{e\in E\ |\ i\in e\}$. The cardinality $|E_i|$ of the set $E_i$ is defined as the degree of the vertex $i$, which is denoted by $d_i$. A hypergraph is \textit{regular of degree $r$} if $d_1=\cdots=d_n=r$.

If  $|e_i\cap e_j|=0$ or $s$ for all edges  $e_i\neq e_j$, then $G$ is called an \textit{$s$-hypergraph}.
An ordinary graph  is a $2$-uniform $1$-hypergraph. Note that $1$-hypergraph here is also called \textit{linear hypergraph} in \cite{Bretto-hypergraph}, and in this paper we shall say linear hypergraph when we mean $1$-hypergraph.  We
assume that $G$ is simple throughout the paper, i.e. $e_i\neq e_j$ if $i\neq j$. In a hypergraph, two vertices are said to be \textit{adjacent} if there is an edge that contains both of these vertices. Two edges are said to be \textit{adjacent} if their intersection is not empty. A vertex $v$ is said to be \textit{incident} to an edge $e$ if $v\in e$.  In a hypergraph $G$, a \textit{path of length $q$} is defined to be  a sequence of vertices and edges $v_1,e_1,v_2,e_2,\ldots,v_q,e_q,v_{q+1}$ such that
\begin{enumerate}
   \item $v_1,\ldots,v_{q+1}$ are all distinct vertices of $G$,
   \item $e_1\ldots,e_q$ are all distinct edges of $G$,
   \item $v_r,v_{r+1}\in e_r$ for $r=1,\ldots,q$.
 \end{enumerate}
If $q>1$ and $v_1=v_{q+1}$, then this path is called a \textit{cycle of length $q$}. A hypergraph $G$ is \textit{connected} if there exists a path starting at $v$ and terminating at $u$ for all $v,u\in V$, and is called \textit{acyclic} if it contains no cycle. These definitions can be found in \cite{Berge-hypergraph} and \cite{Bretto-hypergraph}.

\section{Supertrees}\label{sec-tree}

In graph theory, a tree is defined to be a connected graph without cycles. Analogously, we introduce the concept of supertree as follows.
\begin{definition}
A \textit{supertree} is a hypergraph which is both connected and acyclic.
 \end{definition}

A characterization of acyclic hypergraph has been given in Berge's textbook \cite{Berge-hypergraph} and particularly for the connected case is the following result.

\begin{proposition}\cite[Proposition 4, p.392]{Berge-hypergraph}\label{prop-hypergraphacy1}
If $G$ is a connected hypergraph with $n$ vertices and $m$ edges, then it is acyclic if and only if $\sum_{i\in[m]}(|e_i|-1)=n-1$.
\end{proposition}

In particular, if $G$ is a connected $k$-uniform hypergraph with $n$ vertices and $m$ edges, then it is acyclic if and only if $m=\frac{n-1}{k-1}$.

 \begin{proposition}\label{prop-edgenumber}
A  supertree $G$ is a  linear hypergraph. If in addition, $G$ is a $k$-uniform supertree on $n$ vertices, then it  has $\frac{n-1}{k-1}$ edges.
 \end{proposition}
  \begin{proof}
Suppose on the contrary that $G$ is not a linear hypergraph, then there exist two distinct edges $e_i$ and $e_j$ having at least two common vertices, say $\{v_1,v_2\}\subseteq e_i\cap e_j$. Then $v_1,e_i,v_2,e_j,v_1$ would be a cycle of length 2, contradicting that $G$ is acyclic. So $G$ is a $1$-hypergraph or equivalently linear hypergraph.

By Proposition~\ref{prop-hypergraphacy1}, we know that a $k$-uniform supertree on $n$ vertices has $\frac{n-1}{k-1}$ edges.
  \end{proof}

\begin{proposition}\label{prop-cytj} Let $n, \ k$ be positive integers with $n\ge k$. Then there exists a $k$-uniform supertree with $n$ vertices if and only if $n-1$ is a multiple of $k-1$.
\end{proposition}
\begin{proof} The necessity follows from Proposition~\ref{prop-edgenumber}. Now if $n-1$ is a multiple of $k-1$. Let $n'=\frac {n-1}{k-1}+1$, and take $G$ to be the $k$-th power of an ordinary tree $T$ of order $n'$. Then it is easy to verify that $G$ is a $k$-uniform supertree with $n$ vertices.
\end{proof}

\begin{definition}[\cite{HuQishao13}]
Let  $G=(V, E)$ be a $k$-uniform hypergraph with $n$ vertices and $m$ edges.
\begin{itemize}
  \item If there is a disjoint partition of the vertex
set $V$ as  $V=V_0\cup V_1\cup\cdots\cup V_m$ such that $|V_0|=1$ and  $|V_1|=\cdots=|V_m|=k-1$, and $E=\{V_0\cup V_i\ |  i\in[m]\}$, then $G$ is called a hyperstar, denoted by $\mathcal{S}_{n,k}$.
  \item  If $G$ is a path $(v_0,e_1,v_1,\cdots,v_{m-1},e_m,v_m)$ such that the vertices $v_1, \cdots, v_{m-1}$ are of degree two, and all the other vertices of $G$ are of degree one, then $G$ is called a loose path, denoted by $\mathcal{P}_{n,k}$.
\end{itemize}
\end{definition}
Note that both hyperstar and  loose path are supertrees.

\section{The incidence $Q$-tensors of uniform hypergraphs}\label{sec-main}

In \cite{Shao-tensorproduct}, Shao introduced a definition for tensor product and then  Bu et.al \cite{Buchangjiang14} generalized it as follows:
\begin{definition}\label{d7}
Let $\mathcal{A}\in\mathbb{C}^{n_1\times n_2\times\cdots\times n_2}$ and $\mathcal{B}\in\mathbb{C}^{n_2\times\cdots\times n_{k+1}}$ be order $m\geq2$ and $k\geq1$ tensors, respectively. The product
$\mathcal{A}\mathcal{B}$ is the following tensor $\mathcal{C}$ of order $(m-1)(k-1)+1$ with entries
\[
c_{i\alpha_1\cdots\alpha_{m-1}}=\sum_{i_2,\ldots,i_m\in [n_2]}a_{ii_2\cdots i_m}b_{i_2\alpha_1}\cdots b_{i_m\alpha_{m-1}},
\]
where $i\in [n_1]$, $\alpha_1,\cdots,\alpha_{m-1}\in [n_3]\times\cdots\times [n_{k+1}]$.
\end{definition}
Note that by Definition \ref{d7}, now $\mathcal{T}x^{k-1}$ defined
in \eqref{e-produdefn} can be simply written as $\mathcal{T}x$.

Let $G=(V,E)$ be a $k$-uniform hypergraph with vertex set $V=\{v_1,\ldots,v_n\}$, and edge set $E=\{e_1\ldots,e_m\}$. In \cite{Bretto-hypergraph}, the \textit{incidence matrix} of $G$ is defined to be a matrix $R$ whose rows and columns are indexed by the vertices and edges of $G$, respectively. The $(i,j)$-entry of $R$ is
\[
r_{ij}=\left\{
         \begin{array}{ll}
           1, & \hbox{if $v_i\in e_j$;} \\
           0, & \hbox{otherwise.}
         \end{array}
       \right.
\]
Let $\mathbb{I}$ denote the identity tensor of appropriate dimension, e.g., $\mathbb{I}_{i_1\ldots i_k}=1$ if and only if $i_1=\cdots=i_k\in [m]$, and zero otherwise when the dimension is $m$.
 Consider the tensor $R\mathbb{I}R^{T}$. By Definition \ref{d7}, it is a tensor of order $k$ and dimension $n$, whose $(i_1,i_2,\ldots,i_k)$-entry is
\begin{align}\label{eQentry}
&(R\mathbb{I}R^{T})_{i_1,i_2,\ldots,i_k}=\sum_{j=1}^m r_{i_1j}(\mathbb{I}R^{T})_{ji_2\cdots i_k} \nonumber \\
&=\sum_{j=1}^m r_{i_1j}\left(\sum_{l_2,\cdots, l_k=1}^m\mathbb{I}_{jl_2\cdots l_k}r_{i_2l_2}\cdots r_{i_kl_k}\right) \nonumber \\
&=\sum_{j=1}^m r_{i_1j}r_{i_2j}\cdots r_{i_kj}
\end{align}
which is the number of edges $e$ of $G$ such that $i_t\in e$ for all $t=1,\cdots,k$.

Note that then $R\mathbb{I}R^{T}$ is a symmetric tensor of order $k$ and dimension $n$. Consider the homogeneous polynomial $f(x):=x^T(R\mathbb{I}R^{T}x)$, and let $y=R^Tx$, write
$$x(e)=\sum_{i\in e}x_i.$$
Then by $y=R^Tx$ we have $y_j=\sum_{i\in e_j}x_i=x(e_j)$, and so
\begin{equation}\label{e-signLappoly2}
x^T(R\mathbb{I}R^Tx)=x^T(R\mathbb{I}y)=x^T(Ry^{[k-1]})=y^Ty^{[k-1]}=\sum_{j\in[m]}y_j^k=\sum_{j\in[m]}x(e_j)^k
=\sum_{e\in E}x(e)^k.
\end{equation}
Thus it can be seen that when $k$ is even, $f(x)\geq0$ for any $x\in
\mathbb{R}^n$ and so $R\mathbb{I}R^{T}$ is positive semi-definite.

\begin{definition} The incidence $Q$-tensor of the uniform hypergraph $G$, denoted by $\mathcal{Q}^{*}=\mathcal{Q}^{*}(G)$, is defined as $\mathcal{Q}^{*}=R\mathbb{I}R^{T}$.
\end{definition}

From this definition of the incidence $Q$-tensor and the above formula \eqref{eQentry} for the entries of $\mathcal{Q}^{*}$ we also have that

\begin{align}\label{eeigenequcomp}
&(\mathcal{Q}^{*}x)_i=\sum_{i_2,\cdots,i_k=1}^n\mathcal{Q}^{*}_{ii_2,\cdots,i_k}x_{i_2}\cdots x_{i_k}\nonumber\\
&=\sum_{i_2,\cdots,i_k=1}^n\left (\sum_{j=1}^m r_{ij}r_{i_2j}\cdots r_{i_kj}\right )x_{i_2}\cdots x_{i_k}\nonumber\\
&=\sum_{i_2,\cdots,i_k=1}^n\left (\sum_{e\in E_i,i_2\in e,\cdots, i_k\in e}x_{i_2}\cdots x_{i_k}\right )\nonumber\\
&=\sum_{e\in E_i}\left (\sum_{i_2\in e}x_{i_2}\right )\cdots \left (\sum_{i_k\in e}x_{i_k}\right )\nonumber\\
&=\sum_{e\in E_i}x(e)^{k-1}
\end{align}

Xie and Chang \cite{XieChang13signLapHeig} defined the signless Laplacian tensor $T_Q$   of an even uniform hypergraph as the symmetric tensor associated with the polynomial
\begin{equation}\label{e-signLappoly1}
T_Qx^k:=\sum_{e_p\in E}Q(e_p)x^k,\ \ \ \ \forall x\in \mathbb{R}^n,
\end{equation}
where $Q(e_p)x^k=(x_{i_1}+x_{i_2}+\cdots+x_{i_k})^k$, for $e_p=\{i_1,i_2,\ldots,i_k\}\subseteq V$. Comparing Eqs.~\eqref{e-signLappoly1} and \eqref{e-signLappoly2},
clearly $f(x)=T_Qx^k$ and so  $T_Q=\mathcal{Q}^{*}$. That is, our incidence $Q$-tensor coincides with the signless Laplacian tensor $T_Q$ introduced by Xie and Chang, who
defined it for even uniform hypergraph.  Note that we do not put restriction on the parity of $k$, namely,
the incidence $Q$-tensor $\mathcal{Q}^{*}$ applies to both even and odd
uniform hypergraphs.

We prefer not to call $\mathcal{Q}^{*}$ the signless Laplacian tensor of
$G$. In spectral graph theory, the Laplacian matrix and the signless
Laplacian matrix appear together.  The Laplacian matrix is defined
as $L=D-A$, and the signless Laplacian matrix is defined as $Q=D+A$,
where $D$ is the degree diagonal matrix and $A$ is the adjacency matrix of
the graph.    Such a definition was generalized to hypergraphs in
\cite{Qi-Lap-signlessLap} and further studied in
\cite{HuQi14,HuQishao13,HuQiXie,Qishaowang14,shaoshanwu-manu}.  On
the other hand, the tensor $\mathcal{Q}^{*}=R\mathbb{I}R^T$ is closely related to the
incidence matrix $R$ of the hypergraph.
Thus, it is also adequate to be called the incidence $Q$-tensor of that hypergraph.

A  $k$th-order $n$-dimensional tensor $\mathcal{T}=(\mathcal{T}_{i_1i_2\cdots
i_k})$  is called \textit{reducible}, if there exists a nonempty proper
index subset $I\subset [n]$ such that
\begin{equation*}
\mathcal{T}_{i_1i_2\cdots i_k}=0, \,\,\, \  \ \ \ \forall i_1\in I,\  \ \forall i_2,\ldots, i_k\notin I.
\end{equation*}
$\mathcal{T}$ is called \textit{weakly reducible} \cite{Friedland},  if there exists a nonempty proper
index subset $I\subset [n]$ such that
\begin{equation*}
\mathcal{T}_{i_1i_2\cdots i_k}=0, \,\,\, \  \ \ \ \forall i_1\in I,\  \ \mbox {and at least one of the} \ \ i_2,\ldots, i_k\notin I.
\end{equation*}
If $\mathcal{T}$ is not reducible, then $\mathcal{T}$ is called \textit{irreducible}. If $\mathcal{T}$ is not weakly reducible, then $\mathcal{T}$ is called \textit{weakly irreducible}.

It is easy to see from the definition that irreducibility implies weak irreducibility.

It is proved in \cite{Friedland} and \cite{YangYang} that a uniform hypergraph $G$ is connected if and only if its adjacency tensor $\mathcal{A}$ is weakly irreducible (and also if and only if its signless Laplacian tensor $\mathcal{Q}$ is weakly irreducible). Now for the incidence $Q$-tensor $\mathcal{Q}^{*}$, we have the following result.

\begin{lemma}\label{lem-irred}
Let  $G$ be a $k$-uniform hypergraph on $n$ vertices. Then the incidence $Q$-tensor $\mathcal{Q}^{*}$ is irreducible if and only if $G$ is connected.
\end{lemma}
\begin{proof}
Sufficiency. Suppose that $G$ is connected.
 For any nonempty proper subset $I\subset [n]$, choose arbitrarily two vertices $i,j$ such that $i\in I$ and $j\in V\setminus I$. Because $G$ is connected, there exists a path $i_1 (=i),e_1,i_2,e_2,i_3,\ldots,e_q, i_{q+1}(=j)$ such that $i_{l},i_{l+1}\in e_{l}$ for $l=1,\ldots,q$. Since $i_1=i\in I$ and $i_{q+1}=j\in V\setminus I$, clearly there exists some $r(1\leq r\leq q)$ such that $i_r\in I$ and $i_{r+1}\in V\setminus I$. This implies that $\mathcal{Q}^{*}_{i_ri_{r+1}\cdots i_{r+1}}\neq0$ as there exists at least one edge (e.g. $e_{i_r}$) contains both $i_r$ and $i_{r+1}$. Thus $\mathcal{Q}^{*}$ is irreducible.

Necessity. Suppose that $G$ is disconnected. Assume that $G_1$ is a connected component of $G$, with vertex set $V_1=V(G_1)$. Then for any $i_1\in V_1$ and $i_2,\ldots,i_k\in V\setminus V_1$, there exists no edge containing all vertices $i_1,i_2,\cdots,i_k$. So $\mathcal{Q}^{*}_{i_1i_2\cdots i_k}=0$, and thus $\mathcal{Q}^{*}$ is reducible.
\end{proof}
Consequently, if $G$ is connected, then $\mathcal{Q}^{*}$ is irreducible. Since irreducible nonnegative tensor with a nonzero diagonal is primitive, this incidence $Q$-tensor $\mathcal{Q}^{*}$ is also a primitive tensor.

Let $\mathcal{T}$ be a $k$th-order $n$-dimensional nonnegative tensor. The spectral radius of  $\mathcal{T}$ is defined as $\rho(\mathcal{T})=\max\{|\lambda|\ :\ \lambda\ {\rm is\  an\ eigenvalue \  of}\ \mathcal{T}\}$. Part of Perron-Frobenius theorem for nonnegative tensors is stated in the following for reference, and for more details one can refer to a survey \cite{changqizhang}.

\begin{theorem}\label{thmperron}

If $\mathcal{T}$ is a  nonnegative tensor, then $\rho(\mathcal{T})$ is an
eigenvalue with a nonnegative eigenvector $x$
corresponding to it.

 If furthermore $\mathcal{T}$ is weakly irreducible, then $x$ is positive,
and for any eigenvalue $\lambda$ with nonnegative
    eigenvector,  $\lambda=\rho(\mathcal{T})$. Moreover, the
    nonnegative eigenvector is unique up to a constant multiple.
\end{theorem}

Let $\mathbb{R}_{+}^n=\{x\in \mathbb{R}^n \ | \ x\ge 0\}$.

\begin{lemma}[\cite{HuQi13normLap}]\label{lem-radinonnegsym}
Let $\mathcal{T}$ be a symmetric nonnegative tensor of order $k$ and dimension $n$. Then
\begin{equation}\label{eradiusnonsymten}
\rho(\mathcal{T})=\max\{x^T(\mathcal{T}x)\ |\ x\in \mathbb{R}_{+}^n, \sum_{i=1}^nx_i^k=1\}.
\end{equation}
Furthermore, $x\in \mathbb{R}_{+}^n$ with $\sum_{i=1}^nx_i^k=1$ is an eigenvector of $\mathcal{T}$ corresponding to $\rho(\mathcal{T})$ if and only if it is an optimal solution of the maximization problem \eqref{eradiusnonsymten}.
\end{lemma}

From Lemma~\ref{lem-radinonnegsym}, it follows immediately that $\rho(\mathcal{T})$ can also be expressed as follows
\[
\rho(\mathcal{T})=\max\left\{\frac{x^T(\mathcal{T}x)}{x^T(\mathbb{I}x)}\ |\ x\in \mathbb{R}_{+}^n, x\neq0\right\},
\]
where $\mathcal{T}$ and $\mathbb{I}$ have the same order and dimension.  Note that $x^T(\mathbb{I}x)=\sum_{i\in [n]}x_i^k= \ \parallel x\parallel_k^k$. By Theorem~\ref{thmperron}, for an weakly irreducible nonnegative tensor $\mathcal{T}$, it has a unique positive eigenvector $x$ with $\parallel x\parallel_k=1$  corresponding to $\rho(\mathcal{T})$ and then we call $x$ the \textit{principal eigenvector} of $\mathcal{T}$.

\section{The extremal spectral radii of $k$-uniform supertrees and $k$th power hypertrees on $n$ vertices}\label{sec-lar}

\vskip 0.2cm

In this section, we prove our main results Theorem 1 and Theorem 2. For this purpose, we first introduce the operation of ``moving edges" for hypergraphs, together with the two special cases of this operation: the edge-releasing operation and the total grafting operation. We study the perturbation of the three kinds of spectral radii of hypergraphs under these operations: the adjacency spectral radius, the signless Laplacian spectral radius and the incidence $Q$-spectral radius. We show that all these three kinds of spectral radii of supertrees strictly increase under the edge-releasing operation and the inverse of the total grafting operation.

Using these perturbation results, we prove that for all these three kinds of spectral radii, the hyperstar $\mathcal{S}_{n,k}$ attains uniquely the maximum spectral radius among all
$k$-uniform supertrees on $n$ vertices, and we give the exact values of these three kinds of spectral radii of $\mathcal{S}_{n,k}$.
We also determine (in Theorem~\ref{thm-second}) that for all these three kinds of spectral radii, $S^k(1,n'-3)$
attains uniquely the second largest spectral radius among all $k$-uniform supertrees on $n$ vertices (where $n'=\frac {n-1}{k-1}+1$). We also prove that for all these three kinds of spectral radii, the loose path $\mathcal{P}_{n,k}$
attains uniquely the minimum spectral radius among all $k$-th power hypertrees on $n$ vertices.

By Proposition~\ref{prop-cytj}, we know that there exists a $k$-uniform supertree with $n$ vertices if and only if $n-1$ is a multiple of $k-1$. So in this section, we always assume that $n-1$ is a multiple of $k-1$.

\vskip 0.2cm

Recall the Laplacian tensor and signless Laplacian tensor proposed by Qi \cite{Qi-Lap-signlessLap}. Let $G=(V,E)$ be a $k$-uniform hypergraph. The {\em adjacency tensor} of $G$ was defined in \cite{CoopDut12} as the $k$-th order $n$-dimensional tensor $\mathcal A$ whose $(i_1 \ldots i_k)$-entry is:
\begin{eqnarray*}
a_{i_1 \ldots i_k}=\left\{\begin{array}{cl}\frac{1}{(k-1)!}&\mbox{if}\;\{i_1,\ldots,i_k\}\in E,\\0&\mbox{otherwise}.\end{array}\right.
\end{eqnarray*}
Let $\mathcal D$ be a $k$-th order $n$-dimensional diagonal tensor with its diagonal element $d_{i\ldots i}$ being $d_i$, the degree of vertex $i$ in $G$, for all $i\in [n]$. Then $\mathcal L=\mathcal D-\mathcal A$ is the {\em Laplacian tensor} of the hypergraph $G$, and $\mathcal Q=\mathcal D+\mathcal A$ is the {\em signless Laplacian tensor} of the hypergraph $G$.

For a vector $x$ of dimension $n$ and a subset $U\subseteq [n]$, we write
$$x^U=\prod_{i\in U}x_i$$

By \cite{CoopDut12}, we have
$$x^T(\mathcal{A}(G)x)=\sum_{e\in E(G)}kx^e$$
and
$$(\mathcal{A}(G)x)_i=\sum_{e\in E_i(G)}x^{e\setminus \{i\}}$$

Also it is easy to calculate for the signless Laplacian tensor $\mathcal{Q}(G)$ that:

$$x^T(\mathcal{Q}(G)x)=\sum_{\{j_1,\cdots, j_k\}\in E(G)}(x_{j_1}^k+\cdots+x_{j_k}^k+kx_{j_1}\cdots x_{j_k})=\sum_{e\in E(G)}(x^{[k]}(e)+kx^e)$$
where $x^{[k]}(e)=x_{j_1}^k+\cdots +x_{j_k}^k$ for $e=\{j_1,\cdots,j_k\}$, and
$$(\mathcal{Q}(G)x)_i=d_i(G)x_i^{k-1}+\sum_{e\in E_i(G)}x^{e\setminus \{i\}}$$

Now we introduce the operation of $moving$ $edges$ on hypergraphs.

\begin{definition}\label{def-edgemoving}
Let $r\ge 1$, $G=(V,E)$ be a hypergraph with $u\in V$ and $e_1,\cdots,e_r\in E$, such that $u\notin e_i$ for $i=1,\cdots,r$. Suppose that $v_i\in e_i$ and write $e_i'=(e_i\setminus \{v_i\})\cup \{u\}  \  (i=1,\cdots,r)$. Let $G'=(V,E')$ be the hypergraph with $E'=(E\setminus\{e_1,\cdots,e_r\})\cup \{e_1',\cdots,e_r'\}$. Then we say that $G'$ is obtained from $G$ by moving edges $(e_1,\cdots,e_r)$ from $(v_1,\cdots,v_r)$ to $u$.
\end{definition}

\vskip 0.2cm

\noindent {\bf Remark:}

\noindent (1) The vertices $v_1,\cdots,v_r$ need not be distinct. That is, the repetition of some vertices in $v_1,\cdots,v_r$ is allowed.

\noindent (2) Generally speaking, the new hypergraph $G'$ may contain multiple edges. But if $G$ is acyclic and there is an edge $e\in E$ containing all the vertices $u,v_1,\cdots,v_r$, then $G'$ contains no multiple edges.

\begin{theorem}\label{thm-edgemoving} Let $r\ge 1$, $G$ be a connected hypergraph, $G'$ be the hypergraph obtained from $G$ by moving edges $(e_1,\cdots,e_r)$ from $(v_1,\cdots,v_r)$ to $u$, and $G'$ contains no multiple edges.  Then we have:

(1) If $x$ is the principal eigenvector of $\mathcal{A}(G)$ corresponding to $\rho(\mathcal{A}(G))$, and suppose that $x_u\ge \max_{1\leq i\leq r}\{x_{v_i}\}$, then   $\rho(\mathcal{A}(G'))>\rho(\mathcal{A}(G))$.

\noindent (2) If $x$ is the principal eigenvector of $\mathcal{Q}(G)$ corresponding to $\rho(\mathcal{Q}(G))$, and suppose that $x_u\ge \max_{1\leq i\leq r}\{x_{v_i}\}$, then   $\rho(\mathcal{Q}(G'))>\rho(\mathcal{Q}(G))$.

\noindent (3) If $x$ is the principal eigenvector of $\mathcal{Q}^{*}(G)$ corresponding to $\rho(\mathcal{Q}^{*}(G))$, and suppose that $x_u\ge \max_{1\leq i\leq r}\{x_{v_i}\}$, then $\rho(\mathcal{Q}^{*}(G'))>\rho(\mathcal{Q}^{*}(G))$
\end{theorem}
\begin{proof} Let $e_i'=(e_i\setminus \{v_i\})\cup \{u\}  \  (i=1,\cdots,r)$ as in Definition~\ref{def-edgemoving}.

(1). By the hypothesis we obviously have $x^{e_i'}/ x^{e_i}=x_u/x_{v_i}\ge 1$. Thus by using Lemma~\ref{lem-radinonnegsym} and the above expression for $x^T(\mathcal{A}(G)x)$ we have
\begin{align*}
&\rho(\mathcal{A}(G'))-\rho(\mathcal{A}(G))\\
&\geq x^T(\mathcal{A}(G')x)-\rho(\mathcal{A}(G))\\
&= x^T(\mathcal{A}(G')x)-x^T(\mathcal{A}(G)x)\\
&=\sum_{e\in E(G')}kx^e-\sum_{e\in E(G)}kx^e\\
&=k\sum_{i=1}^r(x^{e_i'}-x^{e_i})  \ge 0.
\end{align*}
If the equality holds, then $\rho(\mathcal{A}(G'))=x^T(\mathcal{A}(G')x)$ and so $x$ is the eigenvector of $\mathcal{A}(G')$ corresponding to $\rho(\mathcal{A}(G'))=\rho(\mathcal{A}(G))$ by Lemma~\ref{lem-radinonnegsym}. In this case, using the above expression for $(\mathcal{A}(G)x)_{u}$ and $(\mathcal{A}(G')x)_{u}$, we have
\begin{align*}
0&=(\rho(\mathcal{A}(G'))-\rho(\mathcal{A}(G)))x_{u}^{k-1}\\
&= (\mathcal{A}(G')x)_{u}-  (\mathcal{A}(G)x)_{u}    \\
&= \sum_{e\in E_u(G')}x^{e\setminus \{u\}}-\sum_{e\in E_u(G)}x^{e\setminus \{u\}} \\
&= \sum_{i=1}^rx^{e_i'\setminus \{u\}}>0   \\
\end{align*}
a contradiction.

(2). By the hypothesis we have $x^{[k]}(e_i')-x^{[k]}(e_i)=x_u^k-x_{v_i}^k\ge 0$, and $x^{e_i'}/ x^{e_i}=x_u/x_{v_i}\ge 1$. Thus by using Lemma~\ref{lem-radinonnegsym} and the above expression for $x^T(\mathcal{Q}(G)x)$ we have
\begin{align*}
&\rho(\mathcal{Q}(G'))-\rho(\mathcal{Q}(G))\\
&\geq x^T(\mathcal{Q}(G')x)-\rho(\mathcal{Q}(G))\\
&= x^T(\mathcal{Q}(G')x)-x^T(\mathcal{Q}(G)x)\\
&=\sum_{e\in E(G')}(x^{[k]}(e)+kx^e)-\sum_{e\in E(G)}(x^{[k]}(e)+kx^e)\\
&=\sum_{i=1}^r(x^{[k]}(e_i')- x^{[k]}(e_i))+k\sum_{i=1}^r(x^{e_i'}-x^{e_i})  \ge 0,
\end{align*}
Also by using the above expression for $(\mathcal{Q}(G)x)_{u}$, the strict inequality can be obtained similarly from the following relation:
$$(\mathcal{Q}(G')x)_{u}-  (\mathcal{Q}(G)x)_{u} = (d_{u}(G')-d_{u}(G))x_u^{k-1}+ \sum_{i=1}^rx^{e_i'\setminus \{u\}}>0$$
where $d_{u}(G')-d_{u}(G)=r$.

 (3). By the hypothesis we have:
$$x(e_i')-x(e_i)=x_u-x_{v_i}\ge 0 \qquad (i=1,\cdots, r).$$
Thus we have $x(e_i')^k\ge x(e_i)^k$, since $x$ is a positive vector. Now by Lemma~\ref{lem-radinonnegsym} and the equation~(\ref{e-signLappoly2}) we have
\begin{align*}
&\rho(\mathcal{Q}^{*}(G'))-\rho(\mathcal{Q}^{*}(G))\\
&\geq x^T(\mathcal{Q}^{*}(G')x)-\rho(\mathcal{Q}^{*}(G))\\
&= x^T(\mathcal{Q}^{*}(G')x)-x^T(\mathcal{Q}^{*}(G)x)\\
&=\sum_{e\in E(G')}x(e)^k -\sum_{e\in E(G)}x(e)^k\\
&=\sum_{i=1}^r(x(e_i')^k- x(e_i)^k)\ge 0,
\end{align*}
If the equality holds, then $\rho(\mathcal{Q}^{*}(G'))=x^T(\mathcal{Q}^{*}(G')x)$ and so $x$ is the eigenvector of $\mathcal{Q}^{*}(G')$ corresponding to $\rho(\mathcal{Q}^{*}(G'))=\rho(\mathcal{Q}^{*}(G))$ by Lemma~\ref{lem-radinonnegsym}. In this case, applying eigenvalue equations and equation~(\ref{eeigenequcomp}) to the vertex $u$ in $G'$ and $G$, we find
\begin{align*}
0&=(\rho(\mathcal{Q}^{*}(G'))-\rho(\mathcal{Q}^{*}(G)))x_{u}^{k-1}\\
&= (\mathcal{Q}^{*}(G')x)_{u}-  (\mathcal{Q}^{*}(G)x)_{u}    \\
&= \sum_{e\in E_{u}(G')}x(e)^{k-1} -\sum_{e\in E_{u}(G)}x(e)^{k-1}    \\
&=\sum_{i=1}^rx(e_i')^{k-1}>0 ,   \\
\end{align*}
a contradiction.

\end{proof}

Recall that a linear hypergraph is a hypergraph each pair of whose edges has at most one common vertex. We have proved in Proposition~\ref{prop-edgenumber} that all supertrees are linear hypergraphs.

In a $k$-uniform   linear hypergraph $G$, an edge $e$ is called a \textit{pendent edge} if $e$ contains exactly $k-1$ vertices of degree one. If $e$ is not a pendent edge, then it is also called a \textit{non-pendent edge}.

The following $edge$-$releasing$ operation on linear hypergraphs is a special case of the above defined edge moving operation.

\begin{definition}\label{def-edgereleasing}
Let $G$ be a $k$-uniform  linear hypergraph, $e$ be a non-pendent edge of $G$ and $u\in e$. Let $\{e_1,e_2,\ldots,e_r\}$ be all the edges of $G$ adjacent to $e$ but not containing $u$, and suppose that $e_i\cap e=\{v_i\}$ for $i=1,\ldots,r$. Let $G'$ be the hypergraph obtained from $G$ by moving edges $(e_1,\cdots,e_r)$ from $(v_1,\cdots,v_r)$ to $u$. Then $G'$ is said to be obtained from $G$ by an edge-releasing operation on $e$ at $u$.
\end{definition}

In other words, edge-releasing a non-pendent edge $e$ of $G$ at $u$ means moving all the edges adjacent to $e$ but not containing $u$ from their commom vertices with $e$ to $u$.

Since $e$ is a non-pendent edge of $G$, $e$ contains at least one non-pendent vertex different from $u$. So by the definition of linear hypergraph, there exists at least one edge ``adjacent to $e$ but not containing $u$". This means that edge-releasing operation is a special case of the edge-moving operation in Definition~\ref{def-edgemoving}.

From the above definition we can see that, if $G'$ and $G''$ are the hypergraphs obtained from a $k$-uniform linear hypergraph $G$ by an edge-releasing operation on some edge $e$ at $u$ and at $v$, respectively. Then $G'$ and $G''$ are isomorphic. Also, if  $G$ is acyclic, then $G'$ contains no multiple edges.

\begin{proposition}
Let $G'$ be a hypergraph obtained from a $k$-uniform supertree $G$ by edge-releasing a non-pendent edge $e$ of $G$. Then $G'$ is also a supertree.
\end{proposition}
\begin{proof} Since $G$ is connected, it is easy to see that $G'$ is also connected. Also by the definition of the edge-releasing operation we can see that, $G$ and $G'$ have the same number of edges. Thus we have $|E(G')|=|E(G)|=\frac {n-1}{k-1}$. So by Proposition \ref{prop-hypergraphacy1} we conclude that $G'$ is also a supertree.
\end{proof}

 \begin{theorem}\label{thm-edgerelease}
Let $G'$ be a supertree obtained from a $k$-uniform supertree $G$ by edge-releasing a non-pendent edge $e$ of $G$ at $u$.  Then we have $\rho(\mathcal{A}(G'))>\rho(\mathcal{A}(G))$,    $\rho(\mathcal{Q}(G'))>\rho(\mathcal{Q}(G))$, and $\rho(\mathcal{Q}^{*}(G'))>\rho(\mathcal{Q}^{*}(G))$.
\end{theorem}
\begin{proof} To prove $\rho(\mathcal{A}(G'))>\rho(\mathcal{A}(G))$, take $x$ to be the principal eigenvector of $\mathcal{A}(G)$ corresponding to $\rho(\mathcal{A}(G))$, and take $v\in e$ such that $x_v= \max_{i\in e}\{x_i\}$. Let $G''$ be the supertree obtained from $G$ by edge-releasing the edge $e$ of $G$ at $v$, then $G'$ and $G''$ are isomorphic. But by Definition~\ref{def-edgereleasing}, $G''$ is obtained from $G$ by moving some edges from some vertices of $e$ to $v$. So by Theorem~\ref{thm-edgemoving} we have $\rho(\mathcal{A}(G'))=\rho(\mathcal{A}(G''))>\rho(\mathcal{A}(G))$.

Similarly we can prove $\rho(\mathcal{Q}(G'))>\rho(\mathcal{Q}(G))$ and $\rho(\mathcal{Q}^{*}(G'))>\rho(\mathcal{Q}^{*}(G))$.
\end{proof}

\begin{theorem}\label{thm-hyperstar1}
Let $\mathfrak{T}$ be a  $k$-uniform   supertree on $n$ vertices with $m$ edges (here $m=\frac {n-1}{k-1}$).   Then
\begin{equation*}
\rho(\mathcal{A}(\mathfrak{T}))\leq \rho(\mathcal{A}(\mathcal{S}_{n,k}))
\end{equation*}
and
\begin{equation*}
\rho(\mathcal{Q}(\mathfrak{T}))\leq \rho(\mathcal{Q}(\mathcal{S}_{n,k}))
\end{equation*}
and
\begin{equation*}
\rho(\mathcal{Q}^{*}(\mathfrak{T}))\leq \rho(\mathcal{Q}^{*}(\mathcal{S}_{n,k}))
\end{equation*}
with either one of the equalities holding if and only if $\mathfrak{T}$ is the hyperstar $\mathcal{S}_{n,k}$.
\end{theorem}
\begin{proof} We use induction on the number of non-pendent vertices (vertices with degrees at least two) $N_2(\mathfrak{T})$. If $N_2(\mathfrak{T})=1$, then $\mathfrak{T}$ is the hyperstar $\mathcal{S}_{n,k}$. Now we assume that $N_2(\mathfrak{T})\ge 2$, namely $\mathfrak{T}$ is not a hyperstar. Suppose $x$ and $y$ be two  non-pendent vertices. Then there must be some non-pendent edge $e$ in the path from $x$ to $y$. Let $\mathfrak{T'}$ be the supertree obtained from $\mathfrak{T}$ by edge-releasing the non-pendent edge $e$ of $\mathfrak{T}$. Then by Theorem~\ref{thm-edgerelease} we have $\rho(\mathcal{A}(\mathfrak{T}))< \rho(\mathcal{A}(\mathfrak{T'}))$. On the other hand, we have $N_2(\mathfrak{T'})< N_2(\mathfrak{T})$. So by the inductive hypothesis we have $\rho(\mathcal{A}(\mathfrak{T'}))\leq \rho(\mathcal{A}(\mathcal{S}_{n,k}))$. Combining the above two relations we obtain $\rho(\mathcal{A}(\mathfrak{T}))< \rho(\mathcal{A}(\mathcal{S}_{n,k}))$.

Using the same arguments we can prove the second and the third inequalities.
\end{proof}

Next we determine the supertree with the second largest spectral radius (also for the three kinds of spectral radii).

Let $S(a,b)$ be the ordinary tree with $a+b+2$ vertices obtained from an edge $e$ by attaching $a$ pendent edges to one end vertex of $e$, and attaching $b$ pendent edges to the other end vertex of $e$. Let $S^k(a,b)$ be the $k$th power of $S(a,b)$. We have the following lemma for the comparison of the spectral radii of $S^k(a,b)$ and $S^k(c,d)$ when $a+b=c+d$.

\begin{lemma}\label{lem-second}Let $a,b,c,d$ be nonnegative integers with $a+b=c+d$. Suppose that $a\le b$, $c\le d$ and $a<c$, then we have:
$$\rho(\mathcal{A}(S^k(a,b)))>\rho(\mathcal{A}(S^k(c,d))),$$
and
$$\rho(\mathcal{Q}(S^k(a,b)))>\rho(\mathcal{Q}(S^k(c,d))),$$
and
$$\rho(\mathcal{Q^*}(S^k(a,b)))>\rho(\mathcal{Q^*}(S^k(c,d))).$$
\end{lemma}
\begin{proof} Let $x,y$ be the (only) two non-pendent vertices of $S^k(c,d)$ with the degrees $d(x)=c+1$ and $d(y)=d+1$. Let $G'$ be obtained from $S^k(c,d)$ by moving $c-a$ pendent edges from $x$ to $y$, and $G''$ be obtained from $S^k(c,d)$ by moving $d-a$ pendent edges from $y$ to $x$. Then both $G'$ and $G''$ are isomorphic to $S^k(a,b)$.

On the other hand, it can be verified that at least one of $G'$ and $G''$ will satisfy the condition (1) of Theorem~\ref{thm-edgemoving}. So by Theorem~\ref{thm-edgemoving} we have
$$max (\rho(\mathcal{A}(G')),\rho(\mathcal{A}(G'')))>\rho(\mathcal{A}(S^k(c,d))).$$
Thus we have
$$\rho(\mathcal{A}(S^k(a,b)))=max (\rho(\mathcal{A}(G')),\rho(\mathcal{A}(G'')))>\rho(\mathcal{A}(S^k(c,d))).$$
The other two inequalities can be proved in exactly the same way.
\end{proof}

The following theorem shows that for all these three kinds of spectral radii, $S^k(1,n'-3)$
attains uniquely the second largest spectral radius among all $k$-uniform supertrees on $n$ vertices (where $n'=\frac {n-1}{k-1}+1$).

\begin{theorem}\label{thm-second}
Let $\mathfrak{T}$ be a  $k$-uniform   supertree on $n$ vertices (with $m=n'-1$ edges where $n'=\frac {n-1}{k-1}+1$). Suppose that $\mathfrak{T}\ne \mathcal{S}_{n,k}$, then we have
\begin{equation*}
\rho(\mathcal{A}(\mathfrak{T}))\leq \rho(\mathcal{A}(S^k(1,n'-3))),
\end{equation*}
and
\begin{equation*}
\rho(\mathcal{Q}(\mathfrak{T}))\leq \rho(\mathcal{Q}(S^k(1,n'-3))),
\end{equation*}
and
\begin{equation*}
\rho(\mathcal{Q}^{*}(\mathfrak{T}))\leq \rho(\mathcal{Q}^{*}(S^k(1,n'-3))),
\end{equation*}
with either one of the equalities holding if and only if $\mathfrak{T}\cong S^k(1,n'-3)$.
\end{theorem}
\begin{proof} We use induction on the number of non-pendent vertices $N_2(\mathfrak{T})$. Since $\mathfrak{T}\ne \mathcal{S}_{n,k}$, we have $N_2(\mathfrak{T})\ge 2$. Now we assume that $\mathfrak{T}\ne S^k(1,n'-3)$.

If $N_2(\mathfrak{T})=2$, then the two non-pendent vertices (say, $x$ and $y$) of $\mathfrak{T}$ must be adjacent (otherwise, all the internal vertices of the path between $x$ and $y$ would be non-pendent vertices other than $x$ and $y$, contradicting $N_2(\mathfrak{T})=2$), and so it can be easily verified that $\mathfrak{T}= S^k(c,d)$ for some positive integers $2\le c\le d$ ($2\le c$ since $\mathfrak{T}\ne S^k(1,n'-3)$). So by Lemma~\ref{lem-second} we get the desired results.

If $N_2(\mathfrak{T})\ge 3$, let $x,y$ be two non-pendent vertices of $\mathfrak{T}$. Let $x,e_1,x_1,\cdots, e_r,y$ be a path from $x$ to $y$. Let $\mathfrak{T}_1$ be obtained from $\mathfrak{T}$ by moving all the edges incident with $x$ (except $e_1$) to $y$, and $\mathfrak{T}_2$ be obtained from $\mathfrak{T}$ by moving all the edges incident with $y$ (except $e_r$) to $x$. Then both $\mathfrak{T}_1$ and $\mathfrak{T}_2$ are still supertrees (since they are still connected, and have the same number of edges as $\mathfrak{T}$), and we have
$$2\le N_2(\mathfrak{T}_i)=N_2(\mathfrak{T})-1<N_2(\mathfrak{T}) \qquad (i=1,2).$$
So by induction and Theorem~\ref{thm-edgemoving} (since at least one of $\mathfrak{T}_1$ and $\mathfrak{T}_2$ will satisfy the condition (1) of Theorem~\ref{thm-edgemoving}) we have
$$\rho(\mathcal{A}(\mathfrak{T}))<max (\rho(\mathcal{A}(\mathfrak{T}_1)),\rho(\mathcal{A}(\mathfrak{T}_2)))\leq \rho(\mathcal{A}(S^k(1,n'-3)))$$
Using the same arguments we can prove the second and the third inequalities.
\end{proof}

Next we consider the minimal problems for these three kinds of spectral radii. By introducing the operation of $total$ $grafting$ and studying the perturbation of the spectral radii under this operation, we are able to determine that the loose path $\mathcal{P}_{n,k}$ attains uniquely the minimum spectral radius among all $k$-th power hypertrees on $n$ vertices.

A path $P=(v_0,e_1,v_1,\cdots,v_{p-1},e_p,v_p)$ in a $k$-uniform hypergraph $H$ is called a $pendent$ $path$ (starting from $v_0$), if all the vertices $v_1, \cdots, v_{p-1}$ are of degree two, the vertex $v_p$ is of degree one, and all the $k-2$ vertices in the set $e_i\setminus \{v_{i-1},v_i\}$ are of degree one in $H$ ($i=1,\cdots,p$).

\begin{definition}\label{def-grafting}
Let  $G$ be  a connected  $k$-uniform linear hypergraph and $v$ be a vertex of $G$. Let $G(v;p,q)$ be a  $k$-uniform linear hypergraph obtained from $G$ by adding two pendent paths $P=(v,e_1,v_1,\cdots,v_{p-1},e_p,v_p)$ and $Q=(v,e_1',u_1,\cdots,u_{q-1},e_q',u_q)$  at $v$, where $V(P)\cap V(Q)=\{v\}$. Then we say that $G(v;p+q,0)$ is obtained from $G(v;p,q)$ by a $total$ $grafting$ operation at $v$.
\end{definition}

\begin{proposition}\label{prop-grafting}
Let $G(v;p+q,0)$ be the $k$-uniform linear hypergraph obtained from $G$ by adding a pendent path $P=(v,e_1,v_1,\ldots,v_{p-1},e_p,v_p,e_{p+1},\cdots, e_{p+q},v_{p+q})$ at $v$. Let $G_1$ be the hypergraph obtained from $G(v;p+q,0)$ by moving the edge $e_{p+1}$ from $v_p$ to $v$, and let $G_2$ be the hypergraph obtained from $G(v;p+q,0)$ by moving all edges incident to $v$ (except $e_1$) from $v$ to $v_p$. Then both $G_1$ and $G_2$ are isomorphic to $G(v;p,q)$.
\end{proposition}
\begin{proof}
The proof of this result is obvious.
\end{proof}

\begin{theorem}\label{thm-retotateedgeop}
Let $G(v;p,q)$ and $G(v;p+q,0)$ be defined as above (where $G$ is connected). If both $p$ and $q$ are not zero, then
\begin{equation*}
\rho(\mathcal{A}(G(v;p,q)))> \rho(\mathcal{A}(G(v;p+q,0))).
\end{equation*}
and
\begin{equation*}
\rho(\mathcal{Q}(G(v;p,q)))> \rho(\mathcal{Q}(G(v;p+q,0))).
\end{equation*}
and
\begin{equation*}
\rho(\mathcal{Q}^{*}(G(v;p,q)))> \rho(\mathcal{Q}^{*}(G(v;p+q,0))).
\end{equation*}
\end{theorem}
\begin{proof} Let $G_1$ be the hypergraph obtained from $G(v;p+q,0)$ by moving the edge $e_{p+1}$ from $v_p$ to $v$, and let $G_2$ be the hypergraph obtained from $G(v;p+q,0)$ by moving all edges incident to $v$ (except $e_1$) from $v$ to $v_p$. Then both $G_1$ and $G_2$ are isomorphic to $G(v;p,q)$.  Since $G$ is connected, $G(v;p+q,0)$ is also connected and so we can assume that  $x$ is the principal eigenvector of $\mathcal{A}(G(v;p+q,0))$ corresponding to $\rho(\mathcal{A}(G(v;p+q,0)))$.  Consider the components $x_v, x_{v_p}$  of $x$ corresponding to $v$ and $v_p$. Obviously, either $x_v\geq x_{v_p}$ or $x_v\leq x_{v_p}$. Thus by Theorem~\ref{thm-edgemoving}, we have
$$\rho(\mathcal{A}(G(v;p,q)))=max \{\rho(\mathcal{A}(G_1)), \rho(\mathcal{A}(G_2)) \}> \rho(\mathcal{A}(G(v;p+q,0))),$$
Similarly we can show that
$$\rho(\mathcal{Q}(G(v;p,q)))=max \{\rho(\mathcal{Q}(G_1)), \rho(\mathcal{Q}(G_2)) \}> \rho(\mathcal{Q}(G(v;p+q,0))).$$
and
$$\rho(\mathcal{Q}^{*}(G(v;p,q)))=max \{\rho(\mathcal{Q}^{*}(G_1)), \rho(\mathcal{Q}^{*}(G_2)) \}> \rho(\mathcal{Q}^{*}(G(v;p+q,0))).$$
\end{proof}

The following lemma is about the total grafting operation on ordinary trees.
\begin{lemma}\label{lem-grafting}
Let $T$  be an ordinary tree of order $n$ which is not a path. Then the path $P_n$ can be obtained from $T$ by several times of total grafting operations.
\end{lemma}
\begin{proof}
Let $N_{3}(T)$ be the number of vertices in $T$ with degree at least 3. Then $T\ne P_n \Longleftrightarrow N_{3}(T)\ge 1$. We then use induction on $N_{3}(T)$.

Let $v$ be a vertex of $T$, let $u$ be a vertex with degree at least 3 which is furthest to $v$ (since $N_{3}(T)\ge 1$). Then there are at least $(d(u)-1)$ many pendant paths starting from
$u$. By using $(d(u)-2)$ many total grafting operations at $u$ on these pendant paths, we finally obtain a tree $T^{'}$ of order $n$ with
$N_{3}(T^{'})=N_{3}(T)-1$ (since the vertex $u$ has degree 2 in the new tree $T^{'}$). By using induction on the tree $T^{'}$, we arrive our desired result.
\end{proof}

\begin{theorem}\label{thm-kpowertree}
Let  $T^k$ be the $k$th power of an ordinary tree $T$, defined as in \cite{HuQishao13}. Suppose that $T^k$ has $n$ vertices. Then we have
\begin{equation*}
\rho(\mathcal{A}(\mathcal{P}_{n,k}))\leq \rho(\mathcal{A}(T^k))\leq \rho(\mathcal{A}(\mathcal{S}_{n,k}))
\end{equation*}
and
\begin{equation*}
\rho(\mathcal{Q}(\mathcal{P}_{n,k}))\leq \rho(\mathcal{Q}(T^k))\leq \rho(\mathcal{Q}(\mathcal{S}_{n,k}))
\end{equation*}
and
\begin{equation*}
\rho(\mathcal{Q}^{*}(\mathcal{P}_{n,k}))\leq \rho(\mathcal{Q}^{*}(T^k))\leq \rho(\mathcal{Q}^{*}(\mathcal{S}_{n,k}))
\end{equation*}
where either one of the left equalities holds if and only if  $T^k\cong \mathcal{P}_{n,k}$, and either one of the right equalities holds if and only if  $T^k\cong \mathcal{S}_{n,k}$.
\end{theorem}
\begin{proof}
If $T^k\ne \mathcal{P}_{n,k}$, then $T$ is a tree of order $n'=\frac{n-1}{k-1}+1$ which is not a path. By Lemma~\ref{lem-grafting}, $P_{n'}$ can be obtained from $T$ by several times of total grafting operations. Accordingly, $\mathcal{P}_{n,k}$ can be obtained from $T^k$ by several times of total grafting operations. So by Theorem~~\ref{thm-retotateedgeop}, we have $\rho(\mathcal{A}(\mathcal{P}_{n,k}))< \rho(\mathcal{A}(T^k))$, and $\rho(\mathcal{Q}(\mathcal{P}_{n,k}))< \rho(\mathcal{Q}(T^k))$,  and $\rho(\mathcal{Q}^{*}(\mathcal{P}_{n,k}))< \rho(\mathcal{Q}^{*}(T^k))$.

Since $T^k$ is a supertree, the right inequalities follow immediately as a special case of Theorem~\ref{thm-hyperstar1}.
\end{proof}

It is proved in Theorem 4.1 of \cite{HuQiXie} that
\begin{equation*}
\rho(\mathcal{Q}(\mathcal{S}_{n,k}))= 1+\alpha^*,
\end{equation*}
where $\alpha^*\in (m-1, m]$ is the largest real root of $x^k-(m-1)x^{k-1}-m=0$, and $m=\frac{n-1}{k-1}$ is the number of edges of $\mathcal{S}_{n,k}$.

Now we compute the value of $\rho(\mathcal{Q}^{*}(\mathcal{S}_{n,k}))$.

\vskip 0.2cm

 Recall that an automorphism of a $k$-uniform hypergraph $G$
is a permutation $\sigma$ of $V(G)$ such that $\{i_1,i_2,\ldots,i_k\}\in E(G)$
 if and only if $\{\sigma(i_1),\sigma(i_2),\ldots,\sigma(i_k)\}\in E(G)$, for any $i_j\in V(G)$, $j=1,\ldots, k$. The group of all automorphisms of $G$ is denoted by $Aut(G)$.

 In \cite{Shao-tensorproduct}, Shao introduced the concept of permutational similarity for tensors as follows: for two order $k$ and
 dimension $n$ tensors $\mathcal{A}$ and  $\mathcal{B}$, if there exists a permutation matrix $P=P_{\sigma}$ (corresponding to a permutation $\sigma\in S_n$) such that  $\mathcal{B}=P\mathcal{A}P^T$, then
 $\mathcal{A}$ and  $\mathcal{B}$ are called permutational similar. Note that if $\mathcal{B}=P\mathcal{A}P^T$, then $b_{i_1,\ldots,i_k}=a_{\sigma(i_1),\sigma(i_2),\ldots,\sigma(i_k)}$.
 Shao\cite{Shao-tensorproduct} showed that similar tensors have the same characteristic polynomials and thus have the same spectra.

\begin{proposition}
A permutation $\sigma \in S_n$ is an automorphism of a $k$-uniform hypergraph $G$ on $n$ vertices if and only if  $P_{\sigma}\mathcal{Q}^{*}=\mathcal{Q}^{*}P_{\sigma}$.
\end{proposition}
\begin{proof}
Let $P=P_{\sigma}$ be the permutation matrix corresponding to $\sigma$, and $\mathcal{Q}'=P\mathcal{Q}^{*}P^T$. Then we have
$$\mathcal{Q}'_{i_1,\cdots,i_k}=\mathcal{Q}^{*}_{\sigma(i_1),\sigma(i_2),\cdots,\sigma(i_k)}.$$
So by the definition of automorphism and the associative law of the tensor product we have
\begin{align*}
\sigma \in Aut(G)&\Longleftrightarrow
\mathcal{Q}^{*}_{i_1,\cdots,i_k}=\mathcal{Q}^{*}_{\sigma(i_1),\sigma(i_2),\cdots,\sigma(i_k)}=\mathcal{Q}'_{i_1,\cdots,i_k}
\quad (\forall i_1,\cdots,i_k \in [n])\\
&\Longleftrightarrow \mathcal{Q}^{*}=\mathcal{Q}'=P\mathcal{Q}^{*}P^T\\
&\Longleftrightarrow P\mathcal{Q}^{*}=\mathcal{Q}^{*}P
\end{align*}
\end{proof}

If $x$ is an eigenvector of $\mathcal{Q}^{*}$ corresponding to the eigenvalue $\lambda$, then for each automorphism $\sigma$ of $G$ we have
\[
\mathcal{Q}^{*}P_{\sigma}x=P_{\sigma}\mathcal{Q}^{*}x=\lambda P_{\sigma}x^{[k-1]}=\lambda (P_{\sigma}x)^{[k-1]}.
\]
Thus $P_{\sigma}x$ is also an eigenvector of $\mathcal{Q}^{*}$ corresponding to the eigenvalue $\lambda$. This simple observation leads to what follows.

\begin{lemma}\label{lem-eigenvectorcomponent}
Let $G$ be a connected $k$-uniform hypergraph, $\mathcal{Q}^{*}=\mathcal{Q}^{*}(G)$ be its (irreducible) incidence $Q$-tensor. If $x$ is the principal eigenvector of $\mathcal{Q}^{*}$ corresponding to $\lambda=\rho (\mathcal{Q}^{*})$, then we have:

\noindent(1). $P_{\sigma}x= x$ for each automorphism $\sigma$ of $G$.

\noindent(2). For any orbit $\Omega$ of $Aut(G)$ and each pair of vertices $i,j\in \Omega$, the corresponding components $x_i,x_j$ of $x$  are equal.
\end{lemma}
\begin{proof}
\noindent(1). By hypothesis we have $\mathcal{Q}^{*}x=\lambda x^{[k-1]}$. For each automorphism $\sigma$ of $G$ we have
\[
\mathcal{Q}^{*}P_{\sigma}x=P_{\sigma}\mathcal{Q}^{*}x=\lambda P_{\sigma}x^{[k-1]}=\lambda (P_{\sigma}x)^{[k-1]}.
\]
Thus $P_{\sigma}x$ is also an eigenvector of $\mathcal{Q}^{*}$ corresponding to the eigenvalue $\lambda$.
Since $\mathcal{Q}^{*}$ is nonnegative irreducible, by Theorem~\ref{thmperron} the nonnegative eigenvector of $\mathcal{Q}^{*}$ corresponding to $\lambda=\rho (\mathcal{Q}^{*})$ is unique up to a constant multiple. So $P_{\sigma}x=cx$ for some $c\in \mathbb{R}$. Thus $c^2x^Tx=(x^TP_{\sigma}^T)P_{\sigma}x=x^Tx$, and so $c=1$ since both $P_{\sigma}x$ and $x$ are nonnegative.

The result (2) follows directly from result (1).
\end{proof}

Now we can obtain the value of the incidence Q-spectral radius of the hyperstar as in the following theorem.

\begin{theorem}\label{thm-hyperstarradius}
Let $\mathcal{S}_{n,k}$ be a $k$-uniform hyperstar on $n$ vertices.  Then
\[
\rho(\mathcal{Q}^{*}(\mathcal{S}_{n,k}))=(m^{1/(k-1)}+k-1)^{k-1},
\]
where $m=\frac{n-1}{k-1}$ is the number of edges of $\mathcal{S}_{n,k}$.
\end{theorem}
\begin{proof}
Let  $V_0\cup V_1\cup\cdots\cup V_m$ be the disjoint partition of $V(\mathcal{S}_{n,k})$ such that $|V_0|=1$, $|V_1|=\cdots|V_m|=k-1$ and $E=\{V_0\cup V_i\  |  i=1,\ldots,m\}$. Note that $V_0$ and $V_1\cup\cdots\cup V_m$ are two orbits of automorphism group $Aut(\mathcal{S}_{n,k})$. Let $x$ be the principal eigenvector of $\mathcal{Q}^{*}(\mathcal{S}_{n,k})$. Since $\mathcal{S}_{n,k}$ is connected, by Lemma~\ref{lem-eigenvectorcomponent} we have that the components of $x$ corresponding to vertices in $V_0$ and $V\setminus V_0$ are constant respectively, and let $a$ and $b$ be these common values respectively. By the eigenvalue equation $\mathcal{Q}^{*}(\mathcal{S}_{n,k})x=\rho x^{[k-1]}$ and the equation~\eqref{eeigenequcomp}, where $\rho$ denotes $\rho(\mathcal{Q}^{*}(\mathcal{S}_{n,k}))$ for convenience,  we have
\begin{align*}
\rho a^{k-1}&=m (a+(k-1)b)^{k-1},\\
\rho b^{k-1}&=  (a+(k-1)b)^{k-1}.
\end{align*}
Dividing the first equation by the second equation, we obtain $(\frac {a}{b})^{k-1}=m$. Thus $\frac {a}{b}=m^{1/(k-1)}$. So by the second equation we have
\[
\rho= (\frac {a}{b}+k-1)^{k-1}=(m^{1/(k-1)}+k-1)^{k-1}.
\]
\end{proof}

Next we show that $\rho(\mathcal{A}(\mathcal{S}_{n,k}))=m^{1/k}$. Similarly as in the proof of Theorem~\ref{thm-hyperstarradius}, let $x$ be the principal eigenvector of $\mathcal{A}(\mathcal{S}_{n,k})$. Let $u$ be the center of $\mathcal{S}_{n,k}$ (the unique non-pendent vertex). Let $a=x_u$ and $b$ be the common value of all the other components of $x$. Then by the eigenvalue equation $\mathcal{A}(\mathcal{S}_{n,k})x=\rho x^{[k-1]}$, we have
\begin{align*}
\rho a^{k-1}&=m b^{k-1},\\
\rho b^{k-1}&=  ab^{k-2}.
\end{align*}
where $\rho=\rho(\mathcal{A}(\mathcal{S}_{n,k}))$. From this we solve that $\rho=a/b=m^{1/k}$.

\begin{theorem}\label{thm-hyperstar2}
Let $\mathfrak{T}$ be a  $k$-uniform   supertree on $n$ vertices with $m=\frac {n-1}{k-1}$ edges.   Then
\begin{equation*}
\rho(\mathcal{A}(\mathfrak{T}))\leq m^{1/k},
\end{equation*}
and
\begin{equation*}
\rho(\mathcal{Q}(\mathfrak{T}))\leq 1+\alpha^*,
\end{equation*}
where $\alpha^*\in (m-1, m]$ is the largest real root of $x^k-(m-1)x^{k-1}-m=0$, and
\begin{equation*}
 \rho(\mathcal{Q}^*(\mathfrak{T}))\leq (m^{1/(k-1)}+k-1)^{k-1}
\end{equation*}
where either one of the equalities holds if and only if $\mathfrak{T}$ is the hyperstar $\mathcal{S}_{n,k}$.
\end{theorem}
\begin{proof}
The results follow directly from Theorem~\ref{thm-hyperstar1} and Theorem~\ref{thm-hyperstarradius}, and the fact $\rho(\mathcal{A}(\mathcal{S}_{n,k}))=m^{1/k}$, and the proof of Theorem 4.1 in \cite{HuQiXie}.
\end{proof}

\section{Some other properties and bounds on incidence $Q$-spectral radius}

\vskip 0.2cm

In this section, we first give a characterization of regular hypergraphs in terms of their incidence $Q$-tensors, and then using it to give some upper and lower bounds of the incidence $Q$-spectral radii of uniform hypergraphs.

\begin{proposition}\label{prop-regualrQ}
A $k$-uniform hypergraph $G$ is regular (of degree $r$) if and only if its incidence $Q$-tensor
has an all-1 eigenvector (with corresponding eigenvalue $k^{k-1}r$).
\end{proposition}
\begin{proof}
Let $\mathcal{Q}^{*}=(\mathcal{Q}^{*}_{i_1i_2\cdots i_k})$ be the incidence $Q$-tensor of $G$, and $x=\mathbf{1}=(1,\ldots,1)^T$ be the all-1 vector. From the equation~\eqref{eeigenequcomp}, we have
\[
(\mathcal{Q}^{*}\mathbf{1})_i=\sum_{e\in E_i}x(e)^{k-1}=k^{k-1}d_i,\  \ (\forall \  i=1,\cdots,n).
\]
Thus

$G$ is regular of degree $r\Longleftrightarrow d_1=\cdots =d_n=r \Longleftrightarrow \mathcal{Q}^{*}\mathbf{1}=(k^{k-1}r)\mathbf{1}=(k^{k-1}r)\mathbf{1}^{[k-1]}\Longleftrightarrow \mathbf{1}$ is an eigenvector of $\mathcal{Q}^{*}$ with corresponding eigenvalue $k^{k-1}r$.
\end{proof}

A natural way to bound the spectral radius of a symmetric nonnegative tensor is to utilize the way of spectral radius presented in the form of  maximization problem~\eqref{eradiusnonsymten}.

\begin{theorem}\label{thm-Qradiusbounddeg}
Let $G$ be a $k$-uniform  hypergraph  with maximum degree $\Delta$ and average degree $d$.   Then
\begin{equation}\label{e-Qradiusupplowbound}
d\leq\frac{1}{k^{k-1}}\rho(\mathcal{Q}^{*})\leq \Delta.
\end{equation}
with equality holding in either of these inequalities  if and only if $G$ is regular.
\end{theorem}

\begin{proof}

 Since $\mathcal{Q}^{*}$ is a symmetric nonnegative tensor, we have
\[
\rho(\mathcal{Q}^{*})=\max\left\{x^T(\mathcal{Q}^{*}x)\ |\  x\in \mathbb{R}_{+}^n, \parallel x\parallel_k=1\right\}.
\]
Let $x=(1/\sqrt[k]{n})\mathbf{1}=(1/\sqrt[k]{n})(1,\ldots,1)^T$. Then by equation~\eqref{e-signLappoly2} and the fact that $d=\frac{\sum_{i\in[n]}d_i}{n}=\frac{km}{n}$ we have
\begin{align*}
\rho(\mathcal{Q}^{*})\geq x^T(\mathcal{Q}^{*}x)=\sum_{\{j_1,\cdots,j_k\}\in E}(x_{j_1}+\cdots+x_{j_k})^k=m\frac{k^k}{n}=k^{k-1}d,
\end{align*}
If equality holds, then we have $\rho(\mathcal{Q}^{*})=x^T(\mathcal{Q}^{*}x)$ and so $x$ is a eigenvector of $\mathcal{Q}^{*}$ by Lemma~\ref{lem-radinonnegsym}. Thus all-1 vector $\mathbf{1}$ is an eigenvector of $\mathcal{Q}^{*}$ and so  $G$ is regular by Proposition~\ref{prop-regualrQ}.

Now for the right inequality. Let $x$ be a nonnegative eigenvector of $\mathcal{Q}^{*}$ corresponding to $\rho(\mathcal{Q}^{*})$ with $\parallel x\parallel_k=1$. Then we have
\begin{align*}
&\rho(\mathcal{Q}^{*})=x^T(\mathcal{Q}^{*}x)=\sum_{\{j_1,\cdots,j_k\}\in E}(x_{j_1}+\cdots+x_{j_k})^k\\
&\leq k^{k-1}\sum_{\{j_1,\cdots,j_k\}\in E}(x_{j_1}^k+\cdots+x_{j_k}^k)\\
&=k^{k-1}\sum_{i\in [n]}d_ix_{i}^k\\
&\leq k^{k-1}\Delta\sum_{i\in [n]}x_{i}^k\\
&= k^{k-1}\Delta,
\end{align*}
where the first inequality follows from Jensen's inequality  $(\frac{x_{j_1}+\cdots+x_{j_k}}{k})^k\leq \frac{x_{j_1}^k+\cdots+x_{j_k}^k}{k}$.
If  $\rho(\mathcal{Q}^{*})=k^{k-1}\Delta$, then all inequalities above must be all equalities. Thus $d_1=\cdots=d_n=\Delta$.

Conversely, if $G$ is regular, then $d=\Delta$.  By the inequalities~\eqref{e-Qradiusupplowbound}, both sides become equalities.
\end{proof}

Because $\mathcal{Q}^{*}=R\mathbb{I}R^T$, it firstly attracts us to find the relation between the spectral radii $\rho(\mathcal{Q}^{*})$ and $\rho(RR^T)$. For this purpose, we need the following inequalities.

\begin{lemma}[\cite{Inequalities}]\label{lem-Jensen}
If  $0<r<s$ and $a_1,\cdots,a_k\ge 0$, then we have
\begin{equation}\label{e-Jensen1}
(a_1^s+a_2^s+\cdots+a_k^s)^{1/s}<(a_1^r+a_2^r+\cdots+a_k^r)^{1/r}
\end{equation}
unless all $a_1,\ldots,a_k$ but one are zero, and
\begin{equation}\label{e-Jensen2}
\left (\frac {a_1^s+a_2^s+\cdots+a_k^s}{k}\right )^{1/s}>\left (\frac {a_1^r+a_2^r+\cdots+a_k^r}{k}\right )^{1/r}\quad (a_1,\cdots,a_k>0)
\end{equation}
\end{lemma}

\begin{theorem}\label{thm-Qradiuslowerbound}
Let $G$ be a $k$-uniform ($k\geq3$) connected hypergraph on $n$ vertices, and $\mathcal{Q}^{*}=R\mathbb{I}R^T$ be its incidence $Q$-tensor, where $R$ is the incidence matrix of $G$.   Then
\[
\rho(RR^T)<\rho(\mathcal{Q}^{*})<k^{k-2}\rho(RR^T).
\]
\end{theorem}
\begin{proof}
Let $x$ be a nonnegative eigenvector of $RR^T$ with unit length corresponding to its spectral radius $\rho(RR^T)$, and let $y=x^{[2/k]}$. Then $\sum_{i\in[n]}y_i^k=\sum_{i\in[n]}x_i^2=1$, and by inequality~(\ref{e-Jensen1}) we have
\begin{align*}
&\rho(RR^T)=x^T(RR^T)x=\sum_{\{j_1,\cdots,j_k\}\in E}(x_{j_1}+\cdots+x_{j_k})^2\\
&=\sum_{\{j_1,\cdots,j_k\}\in E}((y_{j_1}^{k/2}+\cdots+y_{j_k}^{k/2})^{2/k})^k\\
&\le \sum_{\{j_1,\cdots,j_k\}\in E}(y_{j_1}+\cdots+y_{j_k})^k\\
&=y^T(\mathcal{Q}^{*}y)\\
&\le \rho(\mathcal{Q}^{*}),
\end{align*}
if equality holds in the last inequality, then $y$ is a positive vector since $G$ connected implies that $\mathcal{Q}^{*}$ is nonnegative irreducible. Then the first inequality must be strict by inequality~(\ref{e-Jensen1}). So we always have $\rho(RR^T)<\rho(\mathcal{Q}^{*})$.

For the second inequality, let  $y$ be the principle eigenvector of $\mathcal{Q}^{*}$ corresponding to its spectral radius $\rho(\mathcal{Q}^{*})$, and let $x=y^{[k/2]}$. Then $\sum_{i\in[n]}x_i^2=\sum_{i\in[n]}y_i^k=1$, and by inequality~(\ref{e-Jensen2}) we have
\begin{align*}
&(y_{j_1}+\cdots+y_{j_k})^k=(x_{j_1}^{2/k}+\cdots+x_{j_k}^{2/k})^k=\left (\left (\frac {x_{j_1}^{2/k}+\cdots+x_{j_k}^{2/k}}{k}\right )^{k/2}\right )^2\cdot k^k\\
<&\left (\frac {x_{j_1}+\cdots+x_{j_k}}{k}\right )^2\cdot k^k=(x_{j_1}+\cdots+x_{j_k})^2\cdot k^{k-2}
\end{align*}
From this inequality we have
\begin{align*}
&\rho(\mathcal{Q}^{*})=y^T(\mathcal{Q}^{*}y)=\sum_{\{j_1,\cdots,j_k\}\in E}(y_{j_1}+\cdots+y_{j_k})^k\\
<&\sum_{\{j_1,\cdots,j_k\}\in E}(x_{j_1}+\cdots+x_{j_k})^2\cdot k^{k-2}=x^T(RR^T)x\cdot k^{k-2}\le \rho(RR^T)\cdot k^{k-2}
\end{align*}
\end{proof}

For the purpose of comparing these bounds in Theorem~\ref{thm-Qradiuslowerbound}, take $G_1$ and $G_2$ be the $k$-uniform $s$-path and $s$-cycle on $n$ vertices respectively, where $1\leq s\leq\frac{k}{2}$ (\cite{Qishaowang14}). Let $R_1$ and $R_2$ denote the incidence matrices of $G_1$ and $G_2$ respectively, and let $m_i$ denotes the number of edges of $G_i$ for $i=1,2$. Then we have
\[
R_1^TR_1=kI+sA(P_{m_1}),\ \ \ R_2^TR_2=kI+sA(C_{m_2})),
\]
where $A(P_{m_1})$ and $A(C_{m_2})$ are the  adjacency matrices of ordinary path and cycle on $m_1$ and $m_2$ vertices, respectively.  Note that $\rho(R_1^TR_1)=k+s\rho(A(P_{m_1}))=k+2s\cos\frac{\pi}{m_1+1}$ and $\rho(R_2^TR_2)=k+s\rho(A(C_{m_2}))=k+2s$. Thus we have
\begin{align*}
k+2s\cos\frac{\pi}{m_1+1}&<\rho(\mathcal{Q}^{*}(G_1))<k^{k-2}(k+2s\cos\frac{\pi}{m_1+1}),\\
k+2s&<\rho(\mathcal{Q}^{*}(G_2))<k^{k-2}(k+2s).
\end{align*}
Generally, these lower bounds are not better than the lower bound $k^{k-1}d$ in Theorem~\ref{thm-Qradiusbounddeg}. However, these upper bounds are better than the upper bound $2k^{k-1}$ in Theorem~\ref{thm-Qradiusbounddeg}, because $2s\leq k$ and so $1+\frac{2s}{k}$ and $1+\frac{2s}{k}\cos\frac{\pi}{m_1+1}$ are not more than $2$.

\end{document}